\documentclass[reqno, a4paper, 11pt]{amsart}
\usepackage[T1]{fontenc}
\usepackage[utf8]{inputenc}
\usepackage[english]{babel}
\usepackage{amsmath}
\usepackage{amssymb}
\usepackage{amsfonts}
\usepackage{amsthm}
\usepackage{bm}
\usepackage{enumerate}
\usepackage[shortlabels]{enumitem}
\usepackage{moreenum}
\usepackage{mathtools}
\usepackage{xcolor}
\usepackage{csquotes}
\usepackage{hyperref}
\usepackage{subfigure}

\numberwithin{equation}{section}

\newtheorem{theorem}{Theorem}[section]
\newtheorem{definition}[theorem]{Definition}
\newtheorem{algorithm}[theorem]{Algorithm}
\newtheorem{lemma}[theorem]{Lemma}
\newtheorem{corollary}[theorem]{Corollary}

\newtheorem{remark}[theorem]{Remark}

\newtheorem{assumption}[theorem]{Assumption}

\newcommand{\ii}{i}
\newcommand{\im}{\operatorname{Im}}
\newcommand{\re}{\operatorname{Re}}

\newcommand{\ma}{\begin{pmatrix}}
	\newcommand{\am}{\end{pmatrix}}

\newcommand{\qq}{\quad \quad}
\def\supp{\mathop{\rm supp}\nolimits}
\def\Im{{\rm Im\,}}

\begin{document}

\title[Inverse resonance Love problem] {Inverse resonance problem for Love seismic surface waves}

\date{\today}

\author[Samuele Sottile]{Samuele Sottile}
\address{Department of Materials Science and Applied Mathematics, Faculty of Technology and Society, Malm\"{o} University, SE-205 06 Malm\"{o}, Sweden, \ samuele.sottile@mau.se}
\address{Centre of Mathematical Science, Faculty of Natural Science, Lund University, SE-223 62 Lund, Sweden, \ samuele.sottile@math.lu.se}

\subjclass[2010]{35R30, 35Q86, 34A55, 34L25, 81U40, 74J25} 
\keywords{inverse problems, resonances, Sturm-Liouville problem, Love surface waves}

\begin{abstract}
In this paper we solve an inverse resonance problem for the half-solid with vanishing stresses on the surface: Lamb's problem. Using a semi-classical approach we are able to simplify this three-dimensional problem of the elastic wave equation for the half-solid as a Schr{\"o}dinger equation with Robin boundary conditions on the half-line. We obtain asymptotic values on the number and the location of the resonances with respect to the wave number.
Moreover, we prove that the mapping from real compactly supported potentials to the Jost functions in a suitable class of entire functions is one-to-one and onto and we produce an algorithm in order to retrieve the shear modulus from the eigenvalues and resonances.

\end{abstract}

\maketitle



\vskip 0.25cm

\section{Introduction.}\label{intro}

\setcounter{equation}{0}
\subsection{Inverse resonance problems, previous results.}

The Love boundary value problem for the elastic isotropic medium in the half-space (see \cite{4authorsPaper})
\begin{equation}\label{love}
-\frac{\partial}{\partial Z} \hat{\mu} \frac{\partial \varphi_2}{\partial Z} + \hat{\mu} \, |\xi|^2 \varphi_2
= \omega^2 \varphi_2 ,
\end{equation}
\begin{equation}\label{love_b}
\frac{\partial \varphi_2}{\partial Z}(0) = 0,
\end{equation}
where $Z \in \left( - \infty, 0 \right]$ is the coordinate with direction normal to the boundary, $\hat{\mu}$ is the density-normalized shear modulus, $\xi$ is the dual of the of the coordinate vector $(x,y)$ parallel to the boundary, $\omega$ is the frequency and $\varphi_2$ is the component of the displacement vector on the $y$ direction. Equation \eqref{love} describes the motion of an infinitesimal element of elastic solid on a direction lying on a plane parallel to the Earth's surface. The boundary condition \eqref{love_b} says that the infinitesimal element has zero normal velocity on the Earth's surface. Equations \eqref{love}-\eqref{love_b} are obtained in \cite{4authorsPaper} after decoupling the elastic wave equation for infinitesimal solid and using the semiclassical limit. A change of variable (see Section \ref{Semiclass descr}) in equations \eqref{love}--\eqref{love_b} leads to a Schr{\"o}dinger equation 
\begin{equation}
-u''+Vu= k^2 u, 
\end{equation}
with Robin boundary condition 
\begin{equation}
u'(0) + h u(0)= 0.
\end{equation}
The current literature lacks a thorough treatment of inverse problems for Robin boundary condition, which, because of the seismology application,  earns much more interest, besides being more challenging than the Dirichlet case, which is usually treated in the literature.

In relation to seismology, this means reconstructing the parameters that determine the elasticity of the medium in the interior of the Earth from measurements performed on the boundary of the Earth's surface, which are, for example, the frequencies or the wave numbers of surface waves (eigenvalues and resonances). The Earth is a compact domain, but, for simplification, we consider it as a flat half space $\mathbb{R}^2 \times \left(-\infty, 0\right]$ prescribed with some boundary conditions rendering the problem an exterior boundary value problem.

In this paper we obtain number and location of the wave number of the resonances in the asymptotic regime. Moreover, we are able to reconstruct one of the Lam{\'e} parameters, the shear modulus, which describe the elasticity of the medium at a certain depth of the Earth's interior. The number and location of the wave number of the resonance can be interesting experimentally in order to focus the attention on a certain range of wave number or wave velocities. The reconstruction of the shear modulus from the wave numbers of the Love waves is very important for determining the material present in a certain region of the Earth's interior that we want to probe.

 There are only a few examples of complete characterizations of inverse resonance problems, for instance by Korotyaev (see \cite{Korotyaev}), who solved it on the half-line for compactly supported potentials with Dirichlet boundary condition, or Christiansen, who solved it on the whole line for step-like potentials (see \cite{Christiansen2005}), using some results from an earlier paper of Cohen--Kappeler (\cite{CohenKappeler1985}). Some other examples are \cite{Borthwick, KorotyaevRotation, KorotyaevDirac}. In \cite{Gesztesy} is shown a characterization result for real integrable potentials in the Schr{\"o}dinger operator on the half line using the Krein spectral shift function, which is connected to the scattering phase $\delta(k)$ through the Birman-Krein formula. In \cite{IantchenkoJacobi} asymptotical values of resonances are obtained for the periodic Jacobi operator with finitely supported perturbations. There are some other examples of inverse problems in Seismology. For example, in \cite{Iantchenko} they analyze an inverse spectral problem for the semiclassical Rayleigh operator starting with spectral data being the Weyl-Titchmarsh matrix. In this paper, we reconstruct the semiclassical differential Love operator from data being eigenvalues and resonances. 	It is important to stress that we study an inverse resonance problem, where the set of data is limited only to eigenvalues and resonances.  This leads to data (eigenvalues and resonances), which are easily obtained in the laboratory from scattering cross sections as opposed to other data like scattering functions, normalizing constants or Weyl-Titchmarsh function (or matrix), which cannot be obtained directly from laboratory measurements (see \cite{Brown_Weikard}).

Resonances describe the oscillation and the decay of waves on non-compact domains and, likewise for the eigenvalues, they can be computed explicitly only in very few cases, such as the Eckart barrier potential. In general, it is only possible to determine the distribution of the resonances asymptotically, similar to the Weyl law for the eigenvalues.  

What we call resonances in this paper differs from what physicists usually describe as resonances. In particular, we study the mathematical resonances in terms of the wave number $\xi$ and not the frequency $\omega$, where the latter would lead to a wave with amplitude decreasing in time with an exponential rate given by the imaginary part of the frequency. Hence, these would be spatial resonances, where the amplitude exponentially decays or increases in space with rate involving the imaginary part of the wave number. These resonances would precisely be the so-called Regge poles, which are resonances with respect to the angular variables, in the case of a spherical domain (see \cite{Regge}).

\subsection{Semiclassical description of the Love waves system.}\label{Semiclass descr}

Starting from equations \eqref{love}--\eqref{love_b} we can apply a change of variables so that the resulting boundary value problem assumes a Schr{\"o}dinger-type form. Classical ways to transform the Love problem are the \textit{calibration transform} and the \textit{Liouville transform}. By those transforms, we get a Schr{\"o}dinger-type problem with Robin boundary condition with energy $k^2$, that is related to the usual energy $\omega^2$ by $k^2= \frac{\omega^2}{\hat{\mu}_I} - \xi^2$. The difference between the two types of transforms is that in the former we obtain a potential depending on the shear modulus $\mu$ and the wave number $\xi$, whereas in the latter we obtain a potential depending on $\mu$ and the frequency $\omega$.  Once we have performed the calibration transform, we need to solve an inverse resonance  Schr{\"o}dinger problem with Robin boundary condition, where the resonances are the poles of the resolvent with respect to the parameter $k$ (or $\xi$).

This will be done by extending the result of Korotyaev for Dirichlet boundary condition (see \cite{Korotyaev}) to the mathematically more challenging case of Robin boundary condition. 

The main goal of this paper is to retrieve the shear modulus $\hat{\mu}$ ($\hat{\mu}=\mu/\rho$, with $\rho$ being the density)  as we do in Theorem \ref{Theorem on Lame parameter mu}. Theorem \ref{Theorem on Lame parameter mu} is an application of a characterization (see Theorem \ref{Injectivity of P to W map}) between a class $W_{x_I}$ of Jost functions (see Definition \ref{Class of Jost function}) and a class $\mathbb{V}_{x_I}$ of potentials (see Definition \ref{Class of potentials}).

We also obtain new direct results for the resonances consisting in the asymptotics of the counting function (Theorem \ref{Levinson for our case}) and the estimates of the resonances and their forbidden domain (Corollary \ref{Forbidden domain}), which are similar to the results in the Dirichlet case (\cite{Korotyaev}). 

We make some simplifying assumption in the following.
\begin{assumption}[Homogeneity]\label{Homogeneity assumption}
	We assume that below a certain depth $Z_I$, the medium is homogeneous, so the shear modulus is constant
	\begin{equation}\label{homogeneity in depth}
	\hat{\mu} \left(Z\right) = \hat{\mu} \left(Z_I \right):=\hat{\mu}_I  \qquad \text{for  } Z \leq Z_I.
	\end{equation}
\end{assumption}

We perform the calibration substitution in \eqref{love}--\eqref{love_b} as
\begin{equation*}
\varphi_2=\frac{1}{\sqrt{\hat{\mu}}} u,\quad  \frac{d}{dZ} \left( \hat{\mu} \frac{d}{dZ} \varphi_2 \right)= \frac{1}{4} \hat{\mu}^{-\frac{3}{2}} (\hat{\mu}')^2u-\frac{1}{2} \hat{\mu}^{-\frac{1}{2}}\hat{\mu}'' u+\hat{\mu}^\frac{1}{2}u''
\end{equation*}
and we get
\begin{equation*}
u''-|\xi|^2u=\left[\frac12\frac{\hat{\mu}''}{\hat{\mu}}-\frac14  \left( \frac{\hat{\mu}'}{\hat{\mu}} \right)^2 -\frac{1}{\hat{\mu}}\omega^2\right]u. 
\end{equation*}
We set the quasi momentum $k:=\sqrt{\frac{\omega^2}{\hat{\mu}_I}-|\xi|^2}$ and  
\begin{equation}\label{Definition of the potential}
V=\frac{1}{2}\frac{\hat{\mu}''}{\hat{\mu}}-\frac{1}{4}  \left( \frac{\hat{\mu}'}{\hat{\mu}} \right)^2 -\frac{1}{\hat{\mu}}\omega^2+\frac{1}{\hat{\mu}_I}\omega^2= \frac{(\sqrt{\hat{\mu}})''}{ \sqrt{\hat{\mu}}}-\frac{1}{\hat{\mu}}\omega^2+\frac{1}{\hat{\mu}_I}\omega^2,
\end{equation}
where $\hat{\mu}_I:=\hat{\mu}(Z_I)$ is the value of the shear modulus at the depth $Z_I$, below which the medium is homogeneous. By Assumption \ref{Homogeneity assumption}, $\hat{\mu} (Z) = \hat{\mu}_I $ constant for $Z \leq Z_I$, hence also the derivatives $\hat{\mu}'$ and $\hat{\mu}''$ vanish for $Z \leq Z_I$. This implies that the potential $V$ has compact support and depends only on $Z$ as we fixed $\omega$ and let our spectral parameter $\xi$ vary. In this way the potential $V= V_{\omega}$ can be parametrized by $\omega$ and the resonances are considered in terms of $\xi$.
\begin{remark}
We will assume that the potential $V \in \mathbb{V}_{x_I}$ (Definition \ref{Class of potentials}), that implies the shear modulus $\hat{\mu}$ to be constant below the depth $Z_I$ and to be different than $\hat{\mu}_I$ in an interval of type $\left(Z_I, a+Z_I\right)$ for $a>0$.
\end{remark}

The Love scalar equation takes the following form:
\begin{equation}\label{Schrod equation}
-u''+Vu= \lambda u, \qquad \lambda=k^2,
\end{equation}
with corresponding boundary condition that becomes of Robin type after the transformation
\begin{equation}\label{robin bc}
u'(0) + h u(0)= 0, \qquad h= - \frac{1}{2} \frac{\hat{\mu}'(0)}{\hat{\mu}(0)}.
\end{equation}
To resemble the classical formulation, we make the substitution $Z=-x$, which leads the domain to become $\left[ 0 , +\infty \right)$ and we study the problem in terms of $k$. In our case, the potential of the Schr{\"o}dinger operator is real because we are considering an elastic medium. In the case of an inelastic medium, we would have a complex potential that  implies the loss of part of the energy which is converted into heat. We make a self-adjoint realization in $L^2(\mathbb{R}_+)$ of the operator in \eqref{Schrod equation} due to the boundary condition (see \cite{Chadan}). Then the operator appearing on the left hand side of \eqref{Schrod equation}  prescribed with the domain
\begin{align}\label{domain of self adjoint operator}
D=\{ u \in H^2\left[0, + \infty\right) : u'(0) + h u(0) = 0\}
\end{align}
and the $L^2$ inner product is self-adjoint.

\subsection{The Riemann surface}

The presence of a square root in the definition of $k$ suggests that we should define a Riemann surface for $\xi$ in order for $k$ to be a single-valued holomorphic function there. We make an analytic continuation of the real positive variable $|\xi|$ to the whole complex plane and we define the new complex variable as $\xi$. Let $k_{\omega}(\xi)= \ii \sqrt{  \xi^2 - \frac{\omega^2}{\hat{\mu}_I}}$ and define the Riemann surface $\Omega$ of $k_{\omega}(\xi)$ by taking two sheets of the complex plane with cuts along $\ii \mathbb{R}  \cup  \left[- \frac{\omega}{\sqrt{\hat{\mu}_I}}, \frac{\omega}{\sqrt{\hat{\mu}_I}}\right]$, $\Omega_{+}$ called physical sheet and $\Omega_{-}$ called unphysical sheet, and gluing them in a crosswise way. On the one hand, $\Omega_{+}$ is called physical sheet as all the $\xi$ on this sheet correspond to $k$ with positive imaginary part, which lead to a physical $L^2$ solution. On the other hand, $\Omega_{-}$ is called unphysical sheet as all the $\xi$ on there correspond to $k$ with negative imaginary part, which give rise to a non $L^2$ solution.

We choose a determination of $k_{\omega}(\xi)$ by picking the branch of the square root so that $k_{\omega}(\xi) \in \ii \mathbb{R}_+$, when $\xi \in \Omega_{+}$.
The function $k_{\omega}(\xi)$ becomes single-valued holomorphic on the Riemann surface $\Omega$ and with non-zero derivative everywhere, hence it is a \textit{conformal mapping}. 
The quasi momentum $k_{\omega}(\xi)$ satisfies the following properties
\begin{align*}
&k_{\omega} \left(  \left[0-\ii 0, \frac{{\omega}}{\sqrt{\hat{\mu}_I}}-\ii 0\right)  \right) = \left(0, \frac{{\omega}}{\sqrt{\hat{\mu}_I}}\right], \\
&k_{\omega} \left(  \left[0+\ii 0, \frac{{\omega}}{\sqrt{\hat{\mu}_I}}+\ii 0\right)  \right) =  \left[- \frac{{\omega}}{\sqrt{\hat{\mu}_I}},0\right),
\end{align*}
\begin{align*}
&k_{\omega}(\pm \xi)= \ii \xi + O(|\xi|^{-1}) \qq \qq &\xi \in \Omega_{+}, \re \xi \geq 0, \nonumber \\
&k_{\omega}(\pm \xi) = -\ii \xi + O(|\xi|^{-1}) \qq \qq &\xi \in \Omega_{-}, \re \xi \geq 0, \nonumber
\end{align*}
and also
\begin{equation}\label{properties of complex conjugate of k}
k_{\omega}(\xi) = - \overline{k_{\omega}(\overline{\xi})}  =  k_{\omega}(-\xi) \qquad \text{for } \xi \in \Omega_{\pm}.
\end{equation}
In \eqref{properties of complex conjugate of k} the conjugation is made through paths non intersecting the cuts. The reflection is made by paths that cross the cuts as in Figure \ref{Reflection in Love case}. Hence, when we pass the first cut on the imaginary axis we get to the sheet $\Omega_{-}$ and when we pass through the cut $\left(-\frac{{\omega}}{\sqrt{\hat{\mu}_I}}, \frac{{\omega}}{\sqrt{\hat{\mu}_I}}\right)$ we come back to the original sheet.
From \eqref{properties of complex conjugate of k}, we see that $k_{\omega}(\xi) $ is an even function on each single sheet. We can see that for the Rayleigh problem in \cite{Iantchenko} are used the cuts at the same points of the complex plane for $q_S$, although the Riemann surface seemed defined in a different way leading to $q_S$ being an odd function, while $k$ is an even function here. 
\begin{figure}
\centering
\includegraphics[scale=1.2]{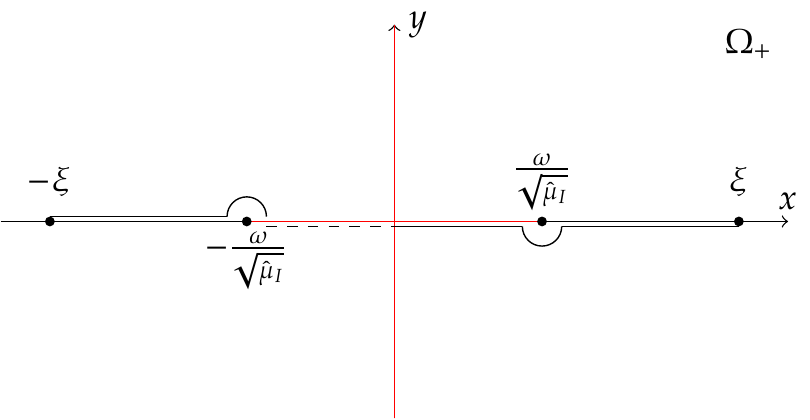}
\caption{Reflection from $\xi$ to $-\xi$ in the physical sheet $\Omega_{+}$. The dashed line represents a path in the unphysical sheet $\Omega_{-}$. The red lines represent the cuts of the Riemann sheets.}
\label{Reflection in Love case}
\end{figure}

\subsection{Cartwright class functions}

In this subsection we give some definitions and results from complex analysis that will be useful later on (see \cite[Chapter 3]{Koosis1988} and \cite[Chapter 1]{Levin}).
\begin{definition}[Exponential type function]
An entire function $f(z)$ is said to be of exponential type if there are real-valued constants $\alpha$, $o$ and $A$ such that 
\begin{align}\label{Exponential type}
|f(z)| \leq A e^{\alpha |z|^o}
\end{align}
for $z\to \infty$ in the complex plane. The infimum of the $o$  and $\alpha$ such that \eqref{Exponential type} is satisfied are called respectively \textit{order} and \textit{type} of the exponential type. 
\end{definition}
In the following we present the Hadamard factorization theorem from \cite[page 279]{Conway}, which we will be fundamental for our analysis.

\begin{theorem}[Hadamard factorization]\label{Hadamard theorem}
Let $f(z)$ be entire of finite order $\rho$ and denote by $a_n$ the sequence of its zeros $\neq 0$ (with multiplicity counted by repetition), so arranged that
\begin{align*}
0< |a_1| \leq |a_2| \leq |a_3| \leq ... \, .
\end{align*}
Then
\begin{equation}\label{product Hadamard theorem}
f(z) = z^m e^{g(z)} \prod_{n} E_p\left(\frac{z}{a_n}\right) 
\end{equation}
where $g(z)$ is a polynomial of degree $q$, $q \leq \rho$, $m$ is the multiplicity of $z=0$ as a zero of $f$ and 
\begin{align*}
E_p(z) = 
\begin{cases}
(1-z) \qq \qq \qq &p=0 \\
(1-z) e^{\frac{z}{1} + \frac{z^2}{2} + ... + \frac{z^p}{p}} &p\neq0
\end{cases}
\end{align*}
with $p=\left[\rho\right]$ being the integer part of $\rho$.
The product \eqref{product Hadamard theorem} is uniformly convergent on compact subsets of $\mathbb{C}$.
\end{theorem}

%
%

\begin{definition}[Cartwright class]\label{Cartwright class definition}
A function $f$ is said to be in the Cartwright class with indices $\rho_{+}=A$ and $\rho_-=B$, if $f(z)$ is entire, of exponential type, and the following conditions are fulfilled:
\begin{equation}\label{Definition Cartwright class}
\int_{\mathbb{R}} \frac{\log^+ |f(x)| dx}{1 + x^2} < \infty, \quad \rho_+(f) = A, \quad \rho_-(f) = B 
\end{equation} 
where  $\rho_{\pm}(f) \equiv \lim \sup_{y\to \infty} \frac{\log|f(\pm iy)|}{y}$ and  $\log^+(x) =\max \left\lbrace \log x,0 \right\rbrace$.
\end{definition}
Basically, for a function to be of Cartwright class means that it is of exponential order 1, of type $A$ in the upper half-plane and $B$ in the lower half-plane and with positive part of the absolute value of its logarithm in $L^1(\mathbb{R}, \Pi)$, where $\Pi$ is the Poisson measure (see \cite{Poltoratski})
\begin{align}
d\Pi(t) = \frac{dt}{1 + t^2}.
\end{align}
For these functions, the Hadamard formula can be simplified. 
Cartwright class functions are very useful in view of a version\footnote{The standard version of the Paley-Wiener theorem states that the Fourier transform of Hardy space functions on a real line (in the upper half plane) are functions in $L^2(\mathbb{R}_+)$. It can be both stated for the upper half-plane and for the unit disc, which can be mapped into each other through a M{\"o}bius transformation.} of the Paley-Wiener theorem because they can be written as the Fourier transform of a compactly supported function (see Lemma \ref{lemma fourier transform}). Another useful application of the Cartwright class property is the Levinson theorem (see \cite[page 69]{Koosis1988}), which is the counterpart of the Weyl law for the resonances. We denote by $\mathcal{N}_+(r,f)$ the number of zeros of an entire function $f$ with positive imaginary part with modulus $\leq r$, and by $\mathcal{N}_-(r,f)$ the number of zeros with negative imaginary part having modulus $\leq r$.  Moreover, $\mathcal{N}_{\pm}(f) :=\lim_{r \to \infty} \mathcal{N}_{\pm}(r,f)$. The total number of zeros with modulus smaller than $r$ is $\mathcal{N}(r,f) := \mathcal{N}_+(r,f) + \mathcal{N}_-(r,f)$.
\begin{theorem}[Levinson]\label{Levinson}
Let the function $f$ be in the Cartwright class with $\rho_{+}=\rho_{-}=A$ for some $A>0$. Then 
\begin{equation*}
\mathcal{N}_{\pm}(r,f) = \frac{Ar}{\pi}\left(1 + o(1)\right) \qq \text{for } r\to \infty.
\end{equation*}
Given $\delta>0$, the number of zeros of $f$ with modulus $\leq r$ lying outside both of the two sectors $|\arg z |<\delta$, $|\arg z -\pi|< \delta$ is $o(r)$ for large $r$.
\end{theorem}

\section{The scattering problem}

By a \textit{direct problem} (or forward problem) we mean the problem of finding the scattering or spectral data associated with a differential operator in a certain class and all their properties. As we can see from \eqref{love} and \eqref{love_b}, the boundary value problem is determined by $V$, for fixed $h$. Hence, we want to define a suitable class for this pair, such that we can find a mapping from this pair to the scattering data of the problem. 

In this section, we introduce the Jost solution and the Jost function because they are the key ingredients we need to be able to obtain information about the scattering data. Below we give some definitions that are essential for our next results.
\begin{definition}[Bargmann-Jost-Kohn]
We define the \textbf{Bargmann-Jost-Kohn} class of potentials\footnote{For this class of potentials it is possible to write the Jost solution in terms of a transformation operator, as we will see. Potentials in this class are short-range potentials.}, and we denote it by $L_{1,1}$, as all the real functions $V(x)$ such that the potential and its first momentum are integrable
\begin{equation}\label{Bargmann-Jost-Kohn}
\int_{0}^{\infty} | (1 + x) \, V(x)| dx < \infty.
\end{equation}
\end{definition}

\begin{definition}[Class of potentials]\label{Class of potentials}
We denote by $\mathbb{V}_{x_I}$ the class of real potentials $V$ such that $V \in  L^1(\mathbb{R}_+)$, $\supp V \subset \left[0,x_I\right]$ for some $x_I >0$ and for each $\epsilon >0$ the set $\left(x_I-\epsilon, x_I\right) \cap \supp V$ has positive Lebesgue measure. 
\end{definition}

	\begin{remark}
We give these two definitions of classes of potentials $L_{1,1}$ and $\mathbb{V}_{x_I}$ to point out that we could solve the inverse problem with either class of potentials. If we consider  $V \in \mathbb{V}_{x_I}$ then we can reconstruct the potential from  only eigenvalues and resonances as data. Otherwise, if $V \in L_{1,1}$ we can reconstruct the potential from the scattering data, such as the scattering function, the eigenvalues and the $L^2$ norm of the eigenfunctions.
\end{remark}

	\begin{definition}[Jost solution]\label{Jost solution definition}
The \textbf{Jost solutions} $f^\pm$ are the unique solutions to the differential equation \eqref{Schrod equation} that satisfy the following condition
\begin{equation}\label{Jost solution condition}
f^\pm(x,k)=e^{\pm ik x}\qquad \mbox{for} \;\; x>x_I.
\end{equation}
\end{definition}
The radiation condition \eqref{Jost solution condition} tells us that the solution of the differential equation must behave like a plane wave far from the scattering area (for $x>x_I$) and implies uniqueness. Since $V \in L_{1,1}$ then the Jost solution can be rewritten as:
\begin{equation}\label{Jost solution +-}
f^{\pm}(x,k)= e^{\pm i k x} + \int^{\infty}_{x} A(x,t) \, e^{\pm i k t} \, dt,
\end{equation}
where $A(x,t)$ is the kernel of the scattering transformation operator (see \cite[Section 4.2]{levitaninverse} or \cite[Theorem 2.1.3]{FreilingYurko}).

The self-adjointness of the differential operator implies that, for real $k$, $f^-(x,k) = \overline{f^+(x,k)} = f^+(x,-k)$. These properties suggest that we  remove the superscript $+$ and $-$ and set $f^+(x,k)=:f(x,k)$. Accordingly, we will refer to $f^-$ as the conjugate of $f$.
The Wronskian between the Jost solution is independent of $x$ (see \cite{levitaninverse}), hence
\begin{equation}\label{wronskian jost sol}
W(f,\overline{f})=  \lim_{x \to \infty} W(f,\overline{f})(x) = \lim_{x \to  \infty} \left(f \overline{f}' -f' \overline{f} \right) (x) = -2ik.
\end{equation}
By solving \eqref{Schrod equation} with the variation of constants method, we can get a Volterra-type equation for the Jost solution $f(x,k)$ (see \cite[Section 4.2]{levitaninverse})
\begin{equation}\label{Volterra-type eq for f}
f(x,k)= e^{ik x} - \int_{x}^{\infty} \frac{\sin \left[ k(x-t) \right] }{k} V(t) f(t,k) dt.
\end{equation}
In this form, the Jost solution can be naturally expanded as a power series of the potential, by Volterra iteration. 

It is a known fact that the spectrum for the operator \eqref{Schrod equation} with domain \eqref{domain of self adjoint operator} consists of a finite number of purely imaginary and simple eigenvalues in $k$ (see \cite{levitaninverse}).

	\begin{definition}[Jost function]\label{Jost function definition}
We define the \textbf{Jost function} $f_h(k)$ of the Schr{\"o}dinger operator $-\frac{d^2}{dx^2} + V$ in \eqref{Schrod equation} with Robin boundary condition \eqref{robin bc} as the quantity
\begin{equation}\label{Jost function}
f_h(k) = f(0,k) h + f'(0,k)
\end{equation}
where $f(0,k)$ is the Jost solution evaluated at $x=0$.
\end{definition}
	We enumerate the zeros of $f_h$, which are eigenvalues and resonances, as $\left(k_j\right)_{j \in \mathbb{N}}$.

	\begin{definition}[Scattering function]
We define the \textbf{scattering function} $S(k)$ of the problem as the negative of the ratio between the Jost function $f_h(k)$ and the reflected Jost function $f_h(-k)$
\begin{equation}\label{scattering function}
S(k)= -\frac{f_h(-k)}{f_h(k)}=  -e^{2i\delta(k)}.
\end{equation}	
\end{definition}
\begin{remark}
In the case of Dirichlet boundary condition it is usually defined as $\frac{f(0,-k)}{f(0,k)}$, where $f(0,k)$ is the Jost function in the Dirichlet case (see \cite{Korotyaev} or \cite[Section 4.2]{levitaninverse}). In the case of Robin boundary condition, it is usually defined as in \eqref{scattering function} (see \cite{Xiao}).
\end{remark}

	In this subsection we will obtain some properties of the Jost function, that will help us with the direct and inverse results. For the following, we recall the definition of $\mathcal{N}_+(f)$. In the following and throughout we define the (complex) Fourier transform $$\hat{g}(k):= \int_{I} g(x) e^{2ixk} dx,$$ $k\in \mathbb{C}$, for $g \in L^1(I)$ with bounded support, where $I$ is an interval. Moreover, throughout the paper we denote by $\mathbb{C}_{\pm}$ the upper and lower open half, respectively, of the complex plane $\mathbb{C}$.

\begin{definition}[Class of Jost function]\label{Class of Jost function}
We define the class $W_{x_I}$ of Jost functions  as the class of all entire functions $f$ such that:
\begin{enumerate}[I]
\item \label{fourier transform of the Jost function}
$f(k) \neq 0$ for all $k \in \mathbb{R}$ and for some $F \in \mathbb{V}_{x_I}$ the function $f$ is given by
\begin{equation}\label{condition 1 class of Jost function}
f(k) = \ii k \left[ 1 - \frac{1}{2ik} \left( \hat{F}(0) - \hat{F}(k) \right)\right], \qq \hat{F}(k) = \int_{0}^{x_I} F(x) e^{2ixk} dx.
\end{equation}
\item \label{condition 2 Jost function class}
All zeros $k_1,... \, ,k_N$ of the function $f$ in $\mathbb{C}_+$ are simple, belong to $i\mathbb{R}_+$ and satisfy for $n=1,... \, , N= \mathcal{N}_+(f)$ :
\begin{equation}\label{condition 2 class of Jost function}
|k_1 |> |k_2|>... \, >|k_N|>0 \quad \text{and}  \quad f_h(-k_n)(-1)^n < 0.
\end{equation}
\end{enumerate}
\end{definition}
It will be clear later why we call $W_{x_I}$ class of Jost functions. In particular, the goal is to prove a bijection between $\mathbb{V}_{x_I}$ and $W_{x_I}$ (see Theorem \ref{Injectivity of P to W map}). In order to do so, we first have to prove that the Jost function is entire and that it satisfies the two conditions of the definition of the class.

In the next theorem we prove that the Jost solution and the Jost function are entire in $k$ (see also \cite[Lemma 3.1.4.]{Marchenko} for the Dirichlet case and $V \in L_{1,1}$ potential).
\begin{theorem}\label{Analyticity and continuity of Jost function and solution}
For each fixed $x\geq0$, the Jost solution $f(x,k)$ and the Jost function $f_h(k)$ are entire in $k$.
\end{theorem}

\begin{proof}
Recall from \eqref{Volterra-type eq for f} the following Volterra-type equation
\begin{equation*}
f(x,k)= e^{ik x} - \int_{x}^{\infty} \frac{\sin \left[ k(x-t) \right] }{k} V(t) f(t,k) dt.
\end{equation*}
Multiplying \eqref{Volterra-type eq for f} by $e^{-ik x}$ and defining the Faddeev solution as $$\chi (x,k) = f(x,k) e^{-ik x}$$ we get
\begin{equation} \label{definition of chi}
\chi (x,k) = 1 - \int_{x}^{\infty} \frac{ 1 -  e^{-2ik (x-t)} }{2ik} V(t) \chi (t,k) dt.
\end{equation}
Iterating \eqref{definition of chi} gives the series
\begin{align}\label{asympt sum of chi}
\chi (x,k) =\sum_{l=0}^{\infty} \chi ^{(l)}(x,k)
\end{align}
where 
\begin{align*}
\chi ^{(0)}(x,k) = 1
\end{align*}
and
\begin{align*}
\chi ^{(l)}(x,k) = (-1)^l \int_{x}^{\infty} \int_{t_1}^{\infty} \cdots  \int_{t_{l-1}}^{\infty} \prod_{j=1}^{l} \frac{ 1 -  e^{-2\ii k (t_{j-1}-t_{j})} }{2\ii k} V(t_j)   dt_l  \cdots dt_1
\end{align*}
with the convention that $t_0=x$. Moreover, we have the estimate 
\begin{align}\label{estimate on chi l}
&|\chi ^{(l)}(x,k)| \leq \int_{x}^{\infty} \int_{t_1}^{\infty} \cdots \int_{t_{l-1}}^{\infty} \frac{ e^{ ( |\im k| - \im k)\left[ (t_1-x) + (t_2 - t_1) + \cdots (t_l -t_{l-1}) \right]}  }{\left(\max (1,|k|)\right)^{l}}  \nonumber \\
&\omit\hfill$\displaystyle |V(t_1)| \cdots |V(t_l)| dt_1 \cdots dt_l $ \nonumber \\
&= \frac{ e^{ ( |\im k| - \im k) (x_I - x)}}{\left(\max (1,|k|)\right)^{l}} \int_{x}^{x_I} \int_{t}^{x_I} \cdots \int_{t_{l-1}}^{x_I} |V(t_1)| \cdots |V(t_l)| dt_1 \cdots dt_l \nonumber \\
& = \frac{ e^{ ( |\im k| - \im k) (x_I - x)}}{\left(\max (1,|k|)\right)^{l}} \frac{1}{l!} \left( \int_{x}^{x_I} |V(t)| dt \right)^l
\end{align}
where we have used the fact that the potential has compact support, $V(x) = 0$ for $x>x_I$, and in the last passage, we have used that those l-integrals with respect to different variables are equal to $l$ times the product of the integral of the potential divided by $l!$.
Each term of the power series is bounded by the term appearing in \eqref{estimate on chi l}, which leads to a uniformly convergent series on every compact set of $k$. By Weierstrass M-test, also the original series \eqref{asympt sum of chi} converges uniformly and absolutely on every compact set, hence, $	\chi (x,k)$ is entire in $k$. Then, also $f(x,k)=\chi (x,k) e^{ikx}$ is entire in $k$. 

For $f'(x,k)$ we have 
\begin{align*}
&f'(x,k) e^{-ikx} \\
& = ik - \int_{x}^{\infty} \frac{1+e^{-2ik(x-t_1)}}{2} V(t_1) dt_1 \\
&\phantom{=\;}+ \int_x^{\infty} \int_{t_1}^{\infty}  \frac{1+e^{-2ik(x-t_1)}}{2} \frac{1-e^{2ik(t_2 - t_1)}}{2ik}  V(t_1) V(t_2) \chi(t_2,k) dt_1 dt_2
\end{align*}
that for the same argument lead to $f'(x,k)$ being entire in $k$. Therefore $f_h(k)$ is also entire.
\end{proof}

	From the proof of Theorem \ref{Analyticity and continuity of Jost function and solution} we get 
\begin{align}
&\left| \chi(x,k) \right| \leq 1 +\left( \sum_{l=1}^{\infty}  \frac{1}{l!} \left( \frac{\int_{x}^{x_I} |V(t)| dt}{\max (1,|k|)} \right)^l\right) e^{ ( |\im k| - \im k) (x_I - x)} \nonumber \\
&= 1 + \left( \sum_{l=0}^{\infty}  \frac{1}{l!} \left( \frac{\int_{x}^{x_I} |V(t)| dt}{\max (1,|k|)} \right)^l - 1\right) e^{ ( |\im k| - \im k) (x_I - x)}, \label{Estimate of chi in both sheet}
\end{align}
that becomes
\begin{align*}
\left| \chi(x,k) \right| \leq e^{\frac{\int_{x}^{x_I} |V(t)| dt}{\max (1,|k|)}}  e^{ ( |\im k| - \im k) (x_I - x)},
\end{align*}
since $1 \leq e^{ ( |\im k| - \im k) (x_I - x)}$.

	The following theorem is a classical result on the simplicity of the eigenvalues for self-adjoint Schr{\"o}dinger operators with $L_{1,1}$ potentials. We follow a proof similar to \cite[Chapter 4, page 79]{levitaninverse}.

\begin{theorem}\label{Zeros of Jost function are finite, simple and pure imaginary}
Let $V \in \mathbb{V}_{x_I}$, then the zeros of the function $f_h(k)$ in the upper half-plane are all simple.
\end{theorem}
\begin{proof}
We assume towards a contradiction that there exists a root $k_0$ which is not simple: $f_h(k_0)=0$ and $\dot{f}_h(k_0) = 0$, where the dot indicates the derivative with respect to $k$. These two conditions mean that:
\begin{align}\label{multiple roots}
f'(0,k_0)=-f(0,k_0)h \nonumber \\
\dot{f}' (0,k_0) = - \dot{f}(0,k_0) h
\end{align}
We consider the  differential equation satisfied by the Jost solution $f(x,k)$, and its derivative with respect to $k$
\begin{align*}
-f''(x,k) + V(x)f(x,k) = k^2 f(x,k) \\
-\dot{f}''(x,k) + V(x)\dot{f}(x,k) = 2kf(x,k) + k^2 \dot{f}(x,k) . 
\end{align*}
We multiply the first equation by $\dot{f}(x,k)$ and the second by $f(x,k)$ and subtract the second from the first 
\begin{equation*}
-\dot{f}(x,k) f''(x,k) + f(x,k) \dot{f}''(x,k) = -2kf^2(x,k).
\end{equation*}
Then, integrating by parts and considering that 
\begin{align*}
&\lim_{x \to \infty} \left( f(x,k) \dot{f}'(x,k) - f'(x,k)\dot{f}(x,k) \right) = \lim_{x \to \infty} \left( e^{ikx}(ix)(ik)e^{ikx} \right. \\
&\left. -  (ix)e^{ikx} (ik) e^{ikx} \right) =0
\end{align*}
we get
\begin{align*}
&f'(0,k) \dot{f}(0,k) - \dot{f}'(0,k)f(0,k) + \int^{\infty}_{0} f'(x,k) \dot{f}'(x,k) dx \\
&- \int^{\infty}_{0} \dot{f}'(x,k) f'(x,k)  dx=  -2k \int_{0}^{\infty} f^2(x,k) dx 
\end{align*}
which becomes
\begin{equation}\label{important formula for norming constants}
f'(0,k) \dot{f}(0,k) - \dot{f}'(0,k)f(0,k) = -2k \int_{0}^{\infty} f^2(x,k) dx.
\end{equation}
Evaluating this equation at $k_0$ and using the condition of a multiple root \eqref{multiple roots} we get a contradiction, in fact, the left-hand side is zero while the right-hand side is not. 
\end{proof}
\begin{remark}
For the Robin Laplacian in the half-line the eigenvalue $-h^2$ is simple. In the one-dimensional case, we have seen that the simplicity of the eigenvalue is stable under the addition of a small compact perturbation. However, for dimensions $d\geq 2$, this is not true in general. Indeed, if we consider $-\Delta + V$ in $\mathbb{R}^3$ with $V$ being the Coulomb potential, the eigenvalues are all simple and accumulating at zero. However, if we add a potential coming from an external electric field, we will observe a splitting of the eigenvalues which become multiple (Stark effect). The same happens if we add a potential coming from a magnetic field (Zeeman effect).
\end{remark}

	\begin{lemma}[Uniform bounds on Jost function]\label{uniform bounds on Jost function}
Let $V \in \mathbb{V}_{x_I}$. Then the Jost function is  of exponential type and satisfies the following estimates:
\begin{equation}\label{bound first iteration Jost function}
\left|f_h(k)- ik \right| \leq  \left\| V \right\|  e^{(|\im k| - \im k)x_I}  e^{a}
\end{equation}
\begin{align}\label{second iteration Jost function}
& \left|f_h(k) -ik - h + \frac{\hat{V}(0) + \hat{V}(k) }{2}    \right|  \leq  \left[ |h|  + \frac{\left\| V \right\| }{2}  \right] a    e^{(|\im k| - \im k)x_I} e^{a}
\end{align}
where $\hat{V}(k) = \int_0^{x_I} e^{2ikt} V(t) dt$ is the Fourier transform of the potential\footnote{In physics the Fourier transform of the potential would be the scattering amplitude.}  $V$ and $a =\frac{ ||V||}{\max (1,|k|)}$, with $|| V || := \int_{\mathbb{R}} |V(x)| dx$.
\end{lemma}
\begin{proof}
These formulas result from the bound on the Jost solution and Jost function and from the definition of the Jost solution iterating the Neumann series up to the first order.
We first compute the estimate of the Jost solution (step 1), then we do the same for the derivative of the Jost solution (step 2) and finally, we collect those results to get an estimate of the Jost function (step 3).
\begin{itemize}
\item \textit{Step 1.} 
We start from \eqref{asympt sum of chi} and write it as 
\begin{align*}
\chi (x,k) = \sum_{l=0}^{N} \chi^{(l)}(x,k) + R_{N}(x,k) 
\end{align*} 
where $R_{N}(x,k)$ is the remainder defined as
\begin{align*}
R_{N}(x,k) = (-1)^{N+1} \int_{x}^{\infty} \int_{t_1}^{\infty} \cdots  \int_{t_{N}}^{\infty} &\prod_{j=1}^{N+1} \left(\frac{ 1 -  e^{-2\ii k (t_{j-1}-t_{j})} }{2\ii k} V(t_j) \right) \\
&\cdot \chi (t_{N+1},k)  dt_{N+1}   \cdots dt_1
\end{align*}
with $t_0=x$ and with 
\begin{align*}
R_0(x,k) = - \int_{x}^{\infty} \frac{1 - e^{2 \ii k(t_1-x)}}{2\ii k} V(t_1) \chi(t_1,k) dt_1.
\end{align*}
Using \eqref{Estimate of chi in both sheet} we get $|\chi(x,k)| \leq e^{a} e^{(|\im k| - \im k)x_I} $.
From the definition of the Faddeev solution it follows that $\chi(0,k)=f(0,k)$. The second iterate $\chi ^{(1)}(0,k)$  of \eqref{asympt sum of chi} can be written as 
\begin{equation*}
\chi ^{(1)}(0,k) =	-\frac{1}{2ik} \int_0^{x_I} V(t) dt + \frac{1}{2ik} \int_0^{x_I} e^{2ikt} V(t) dt = - \frac{\hat{V}(0) - \hat{V}(k) }{2ik}.
\end{equation*}
Then 
\begin{align*}
&f(0,k) -1 + \frac{\hat{V}(0) - \hat{V}(k) }{2ik} = \sum_{l=2}^{\infty} \chi^{(l)}(x,k),
\end{align*}
so
\begin{align*}
\left| f(0,k) -1 + \frac{\hat{V}(0) - \hat{V}(k) }{2ik} \right| &\leq \sum_{l=2}^{\infty}  \left|  \chi ^{(l)} (x,k) \right| \\
&\leq e^{ ( |\im k| - \im k) (x_I - x)} \left( e^{ a(x) } -1 \right).
\end{align*}
From \eqref{estimate on chi l} we can get an estimate (Gronwall inequality) of the Jost solution. 
The estimates of the Jost solution after respectively zero and one iteration of the Neumann series are
\begin{align}\label{estimates on Jost solution}
&|f(0,k) - 1| = |R_0(0,k)| \leq e^{(|\im k| - \im k)x_I}  a e^{a} 
\end{align}
and
\begin{align}\label{estimates on Jost solution 2}
\left|f(0,k)  - 1 + \frac{\hat{V}(0) - \hat{V}(k) }{2ik} \right| = |R_1(0,k)| &\leq e^{(|\im k| - \im k)x_I} \left( e^{a}-1 - a\right) \nonumber \\
&\leq \frac{a^2}{2} e^{(|\im k| - \im k)x_I} e^{a}.
\end{align}
\item \textit{Step 2.} For $f'(x,k)$ we have
\begin{align*}
&f'(x,k) = ike^{ikx} - \int_{x}^{\infty} \cos \left[k \left(x-t_1\right)\right] V(t_1) f(t_1,k) dt_1
\end{align*}
which can be written as
\begin{align*}
&f'(x,k) e^{-ikx} \\
&= ik - \int_{x}^{\infty} \frac{1+e^{-2ik(x-t_1)}}{2} V(t_1) \chi(t_1,k) dt_1  \\
& = ik - \int_{x}^{\infty} \frac{1+e^{-2ik(x-t_1)}}{2} V(t_1) dt_1 \\
&\phantom{=\;}+ \int_x^{\infty} \int_{t_1}^{\infty}  \frac{1+e^{-2ik(x-t_1)}}{2} \frac{1-e^{2ik(t_2 - t_1)}}{2ik}  V(t_1) V(t_2) \chi(t_2,k) dt_1 dt_2
\end{align*}
Evaluating it at $x=0$ and resolving the first integral, we get
\begin{align*}
&f'(0,k) = ik -\frac{\hat{V}(0) + \hat{V}(k) }{2} + \\
& + \int_0^{\infty} \int_{t_1}^{\infty}  \frac{1+e^{2ikt_1}}{2} \frac{1-e^{2ik(t_2 - t_1)}}{2ik}  V(t_1) V(t_2) \chi(t_2,k) dt_1 dt_2
\end{align*}
from which we derive the estimate
\begin{equation}\label{estimates on derivative Jost solution}
\left| f'(0,k) - ik +\frac{\hat{V}(0) + \hat{V}(k) }{2}  \right| \leq \max (1,|k|) \frac{a^2}{2} e^{(|\im k| - \im k)x_I} e^{a}.
\end{equation}
\item \textit{Step 3.} The estimate of the Jost function at the first order in $k$ can be obtained only considering \eqref{estimates on derivative Jost solution} truncated at the zeroth order and it becomes
\begin{equation*}
\left|f_h(k)- ik \right| \leq  \left\| V \right\| e^{(|\im k| - \im k)x_I}  e^{a}.
\end{equation*}
Instead, using \eqref{estimates on Jost solution} truncated at the zeroth order multiplied by $h$ and using \eqref{estimates on derivative Jost solution} we obtain
\begin{align*}
&\left| f'(0,k) + h f(0,k) - ik +\frac{\hat{V}(0) + \hat{V}(k) }{2} - h \right| \\
& \leq  \left[ |h|  + \frac{\left\| V \right\| }{2}  \right]  a   e^{(|\im k| - \im k)x_I} e^{a}.
\qedhere
\end{align*}
\end{itemize}
\end{proof}
As a corollary of Lemma \ref{uniform bounds on Jost function}, we show the form of the previous estimates on the physical sheet.

\begin{corollary}\label{Corollary asympt expans phys sheet}
Let $V \in \mathbb{V}_{x_I}$, then the Jost function $f_h$ is  of exponential type and satisfies the following asymptotic expansion in the physical sheet $\left( \im k > 0\right)$
\begin{equation}\label{asymptotic expansion in phys sheet}
f_h(k) = \ii k + O(1) \qq \qq \text{for } |k|>1.
\end{equation}
\begin{proof}
In the physical sheet $\im k>0$, $|\im k| - \im k = 0$. Then \eqref{bound first iteration Jost function} becomes
\begin{equation*}
\left|f_h(k)- ik \right| \leq  \left\| V \right\|  e^{a}
\end{equation*}
which for $k>1$ implies
\begin{align*}
\left|f_h(k)- ik \right| \leq  ||V|| e^{\frac{||V||}{|k|}}.
\end{align*}
This yields \eqref{asymptotic expansion in phys sheet}.
\end{proof}
\end{corollary}

We shall use a version of Lemma 2.1 from Korotyaev \cite{Korotyaev} adapted to our setting.
\begin{lemma}\label{Lemma norming constants}
If $V \in \mathbb{V}_{x_I}$ and if $k_1,... \,,k_N \in i\mathbb{R}_+$ are the zeros of the Jost function $f_h(k)$ such that $|k_1|> ...  > |k_N|>0$, then the normalizing constants $m_j$, defined as 
\begin{equation}\label{norming constant}
m_j = \int_0^{\infty} f^2(x,k_j) dx
\end{equation}
satisfy 
\begin{equation}\label{norming constant for eigenvalues}
m_j= - \ii \left[ \frac{  \dot{f}_h(k) }{ f_h(-k)} \right]_{k=k_j}>0, \qquad \qquad \text{for} \; j=1,... \,,N,
\end{equation}
and the following inequalities hold
\begin{equation}\label{inequalities for the derivative of the Jost funct}
\ii (-1)^j \dot{f}_h(k_j) >0, \quad \text{and} \quad (-1)^j f_h(-k_j) <0, \qquad \text{for} \; j=1,... \,,N,
\end{equation}
where the dot denotes the derivative with respect to $k$.
\end{lemma}
\begin{proof}
Since $k_j$ is an eigenvalue, then $f_h(k_j)=0$, which means
\begin{equation}\label{condition from being solution of eq diff}
f'(0,k_j) = -f(0,k_j) h.
\end{equation}
Plugging this formula into the Wronskian \eqref{wronskian jost sol} between $f(x,k)$ and $f(x,-k)$ we get
\begin{equation*}
-2i k_j = f(0,k_j) \left[ f'(0,-k_j) + h f(0,-k_j) \right]
\end{equation*}
and together with formula \eqref{important formula for norming constants} and \eqref{condition from being solution of eq diff} we get
\begin{align*} 
& \int_{0}^{\infty} f^2(x,k_j) dx = \frac{f(0,k_j) \left[ -h \dot{f}(0,k_j) - \dot{f}'(0,k_j) \right]   }{-2k_j }\\
&=\frac{  f(0,k_j) \left[ -h \dot{f}(0,k_j) - \dot{f}'(0,k_j) \right]   }{ - \ii f(0,k_j) \left[ h f(0,-k_j) +  f'(0,-k_j) \right] } = - \ii \left[ \frac{\frac{d}{dk}  f_h(k) }{f_h(-k)} \right]_{k=k_j}.
\end{align*}


We can see that for $k \in i \mathbb{R}_+$, as $|k|\to \infty$, the Jost function $f_h(k)$ tends to $ -\infty$. So the first zero of the Jost function, $k_1$, has negative derivative $i \dot{f}_h(k_1)<0$, consequently the next zero has positive derivative $i \dot{f_h}(k_2)>0$ and so on. This implies the second inequality in \eqref{inequalities for the derivative of the Jost funct}, since the ratio must be positive.
\end{proof}

\begin{remark}
From classical results we know that $f_h(k)$ and $\dot{f_h}(k)$ cannot vanish simultaneously (see \cite{levitaninverse}). Then, we can see that $\dot{f_h}(k_j)$ is different from zero at the $k_j \in \ii \mathbb{R}_+$, zero of the Jost function $f_h(k)$. Then \eqref{norming constant for eigenvalues} makes sense.
\end{remark}
The following lemma makes a connection between the Jost function and the potential (compare with Lemma 2.2. in Korotyaev \cite{Korotyaev}), which makes use of the Paley-Wiener theorem for functions in the Cartwright class.
\begin{lemma}\label{lemma fourier transform}
\
\begin{enumerate}[(i)] 
\item \label{item 1}
Let $f$ be entire, of exponential type, and let $\rho_+ (f) = 0 $ and $\rho_- (f) \leq 2x_I$. 
If the following asymptotic holds
\begin{equation*}
f(k) = \ii k \left[ 1 - \frac{1}{2ik} \left( C - \hat{g}(k) + O(k^{-1}) \right)\right]; \qquad k\to \pm \infty,
\end{equation*}
for some $g\in L^1(0,x_I)$ and some constant $C$, then there exists $F \in L^1(0,x_I)$ such that
\begin{equation}\label{Jost funct in terms of Fourier transf}
f(k) = \ii k \left[ 1 - \frac{1}{2\ii k} (  \hat{F}(0) - \hat{F}(k)   )\right], \qquad k \in \mathbb{C},
\end{equation}
where $\rho_- (f) = 2 \sup \left[ \supp F \right]$.\\
\item \label{item 2}
For each $V \in \mathbb{V}_{x_I}$ there exists $p \in L^1(0,x_I)$ such that 
\begin{equation}\label{Jost function mapped from the potential}
f_h(k) =\ii k (1 + \hat{\zeta}(k) )= \ii k \left[  1 - \frac{1}{2ik} (  \hat{p}(0) - \hat{p}(k)   )\right], \quad \zeta(t):= \int_{t}^{x_I}p(x)dx.
\end{equation}
\end{enumerate}
\end{lemma}
\begin{proof}
\eqref{item 1} is proved in the same way as \cite[Lemma 2.2]{Korotyaev} with slight modifications.

\eqref{item 2}
From Lemma \ref{uniform bounds on Jost function} we can write
\begin{align*}
&f_h(k) = \ii k \left[ 1 - \frac{1}{2 \ii k } \left( \hat{V}(0) - 2 h + \hat{V}(k) + O(k^{-1}) \right) \right] . 
\end{align*}
Then the Jost function can be written as
\begin{equation*}
f_h(k) = \ii k \left[ 1 - \frac{1}{2\ii k} (C - \hat{g}(k) + u(k))\right]
\end{equation*}
where $C = \hat{V}(0) - 2h$, $\hat{g}(k)=   \hat{V}(k)$ and $u(k) =  O(k^{-1})$, when $ k\to \pm \infty$.\\
Using the Paley-Wiener theorem, since the function $u(k)$ is entire, of exponential type and square integrable over horizontal lines, there exist a $v \in L^2_{c}(0,x_I)$ which is the Fourier transform of this function, so $u(k) = \hat{v}(k)$, where $v \in  L^2_{c}(0,x_I) \subset L^1_{c}(0,x_I)$.
Using \eqref{item 1} with $p= g + v$ and $\hat{p}(0) = \hat{g}(0) + \hat{v}(0) =  \hat{V}(0) -2h$ we get \eqref{Jost function mapped from the potential} and integrating by parts we obtain $f_h(k) = \ii k (1 + \hat{\zeta}(k) )$ (see \cite{Korotyaev}).
\end{proof}

Lemma \ref{lemma fourier transform} tells us that if we have a function in the Cartwright class with $\rho_-=2 x_I$, then by Definition \ref{Class of Jost function} of the class $W_{x_I}$ property I is satisfied, but this is not enough to have a bijection (we already proved that if $f$ is in $W_{x_I}$ then is in the Cartwright class) because we  also need property II to be satisfied. 

\section{Direct results}

	In this subsection, we state the direct resonance results for the Love problem in terms of the parameter $k$ and $\xi$. We use the property of the Jost function being in the Cartwright class in order to use the Levinson Theorem \ref{Levinson}. In the following lemma, from estimates of the Jost function obtained in the previous subsection we recover estimates on the resonances, which tell us where they are located in the complex plane.

\begin{lemma}[Resonance-free regions]\label{Corollary resonances}
For any zero $k_n$, $n \geq 1$, of $f_h(k)$, $V \in  \mathbb{V}_{x_I} $, the following estimates are fulfilled:
\begin{equation}\label{Estimate on kn}
|k_n| \leq  C_0 e^{2|\im k_n| x_I}, \qquad C_0 =  \left\|V \right\| e^{ \left\|V \right\|}.
\end{equation}
Additionally, if $V' \in L^1$, then
\begin{equation*}
|k_n|^2 \leq C_1 e^{2|\im k_n| x_I }  \qquad C_1= \left[ \left\|V \right\|^2 + 2 |h|  \left\|V \right\|+\frac{1}{4}\left(|V(0)| + \left\|V' \right\|\right)\right]e^{\left\|V \right\|}.
\end{equation*}

\end{lemma}
\begin{proof}
Here we adapt the proof of Corollary 2.3 of \cite{Korotyaev} to our case. Estimate \eqref{bound first iteration Jost function} evaluated at a zero $k_n$ of $f_h(\cdot , V)$, with $\left|k_n\right|$ large, implies
\begin{equation*}
\left|f_h(k_n) - \ii k_n \right| \leq  \left\| V \right\| e^{2|\im k_n|x_I} e^{\left\| V \right\|}
\end{equation*}
and hence
\begin{equation*}
|k_n| \leq \left\| V \right\| e^{\left\| V \right\|} e^{2|\im k_n|x_I},
\end{equation*}
which gives \eqref{Estimate on kn}. If moreover $V' \in L^1(0,\infty)$, then \eqref{second iteration Jost function} evaluated at $k_n$, with $\left|k_n\right|>1$, implies
\begin{equation*}
\left| f_h(k_n) -\ii k_n \right| \leq \left|  h + \frac{\hat{V}(0) - \hat{V}(k_n) }{2}  \right| + \left(\frac{ \left\|V \right\|^2 }{2|k_n|} + \frac{|h|\left\|V \right\| }{|k_n|}\right) e^{\left\|V \right\|} e^{2|\im k_n|x_I}
\end{equation*}
and hence
\begin{equation}\label{intermediate passage direct result}
\left| -\ii k_n^2 \right| \leq \left| k_n \left(-h + \frac{ \left\|V \right\|}{2}\right) \right| + \left| \frac{\hat{V}(k_n) k_n}{2} \right| + \left(  \frac{\left\|V \right\|^2}{2} + |h|  \left\|V \right\|\right) e^{ \left\|V \right\|} e^{2|\im k_n|x_I}.
\end{equation}
Integrating by parts yields
\begin{align*}
&\hat{V}(k) = \int_{0}^{x_I} e^{2\ii k x}V(x) dx = \left[\frac{1}{2\ii k} e^{2\ii k x} V(x)\right]_{0}^{x_I} - \int_{0}^{x_I} \frac{1}{2\ii k} e^{2\ii k x} V'(x) dx\\
& =-\frac{1}{2 \ii k} \left( V(0) + \int_0^{x_I} V'(x) e^{2\ii kx} dx \right),
\end{align*}
thus
\begin{equation*}
|\hat{V}(k_n)| \leq \frac{1}{2|k_n|} \left[|V(0)| + e^{2|\im k_n|x_I} \left\|V' \right\| \right].
\end{equation*}
Therefore, \eqref{intermediate passage direct result} implies
\begin{align*}
&|k_n|^2\leq \left( \frac{ \left\|V \right\|^2}{2}+|h| \left\|V \right\| + \frac{1}{4}\left(|V(0)| +  \left\|V' \right\|\right)  \right) e^{2|\im k_n| x_I +  \left\|V \right\|} \\
&+ |k_n| \left(|h| + \frac{ \left\|V \right\|}{2}\right) \leq C_1 e^{2|\im k_n| x_I +  \left\|V \right\|} 
\end{align*}
where we have used  \eqref{Estimate on kn} in the last passage.
\end{proof}
As a corollary of Lemma \ref{Corollary resonances}, we  infer the resonance free-regions (forbidden domain) in terms of the wave number $\xi$, with $k^2=\frac{\omega^2}{\hat{\mu}}-\xi^2$.
\begin{corollary}\label{Forbidden domain}
For any $\xi_n=\sqrt{\frac{\omega^2}{\hat{\mu}}-k^2_n}$, $n \geq 1$, where $k_n$ are the zeros of $f_h(k, V)$ with $V \in  \mathbb{V}_{x_I} $ the following estimates are fulfilled:
\begin{equation}\label{Estimate on xi n}
|\xi_n| \leq  C_0 e^{2|\re \xi_n| x_I}, \qquad C_0 =  \left\|V \right\| e^{ \left\|V \right\|}.
\end{equation}
Additionally, 	if $V' \in L^1\left(0,\infty \right)$, then
\begin{equation*}
|\xi_n|^2 \leq C_1 e^{2|\re \xi_n| x_I },  \qquad C_1=\frac{3}{2}  \left\|V \right\|^2 + 2 |h|  \left\|V \right\|+\frac{1}{4}\left(|V(0)| + \left\|V' \right\|\right)e^{\left\|V \right\|}.
\end{equation*}

\end{corollary}
\begin{proof}
The proof follows from Lemma\ref{Corollary resonances} after substituting the resonances satisfying $k_n = -\ii \xi_n + O(1)$ for large $\xi_n$.
\end{proof}

In the following corollary to Theorem \ref{Levinson}, we assume that $\rho_-(f_h)=2x_I$, which we will prove later in Lemma \ref{Cartwright class lemma}.
\begin{corollary}[Number of resonances]\label{Levinson for our case}
Let $V \in \mathbb{V}_{x_I}$. Then
\begin{align*}
\mathcal{N}(r,f_h) = \frac{2 x_I r}{\pi} \left(1 + o(1)\right) \qq \text{for }r \to \infty
\end{align*}
For each $\delta >0$ the number of zeros of the Jost function with real part with modulus $\leq r$ lying outside both of the two sectors $|\arg \xi - \frac{\pi}{2}|<\delta$, $|\arg \xi - \frac{3\pi}{2}|<\delta$ is $o(r)$ for large $r$.
\end{corollary}
\begin{proof}
The result follows from Theorem \ref{Levinson} and the fact that $\rho_-(f_h)=2x_I$.
\end{proof}

\section{The inverse problem}

	The goal of the inverse problem is to reconstruct the potential from given data. In the inverse scattering problem, these data can be, for example, the scattering function in addition to eigenvalues and normalizing constant. In the inverse spectral problem, the data are the spectral data, which could be the Weyl function. In our inverse problem, we want to reconstruct the potential starting from eigenvalues and resonances.

In this subsection we present the first inverse resonance result, where from eigenvalues and resonances we can retrieve $V$ after proving a bijection between the class $W_{x_I}$ and $\mathbb{V}_{x_I}$ (see Theorem \ref{Injectivity of P to W map}), following the result of Korotyaev (see \cite{Korotyaev}). The characterization is made by adapting the Marchenko theorem (see \cite[Chapter 3]{Marchenko}) to our case with Robin boundary condition,
which we state below.

\begin{definition}\label{scattering function in the class}
For $N \in \mathbb{N}$, we define $\mathcal{S}_N$ to be the set of functions $S(k)$ such that
\begin{enumerate}
\item 
$S(k)$ is continuous and satisfies the identities $S(k)=\overline{S(-k)}=S(-k)^{-1}$ for each $k\in \mathbb{R}$.
\item 
$S(k) - 1 =o(1)$, for $|k|\to \infty$, and the function $G(x) =1/2\pi \int_{\mathbb{R}} (S(k)-1)e^{ikx}dk$ satisfies 
\begin{align*}
&G' \in L^1(\mathbb{R}_+,(x+1)dx), \qq \qquad G=G_1 + G_2, \\ 
&G_1 \in L^1(\mathbb{R}_+), \qq  \qquad  G_2 \in L^2(\mathbb{R}_+) \cap L^{\infty}(\mathbb{R}_+).
\end{align*}

\item 
The increment of $S(k)$ and $N$ are related by the following formula
\begin{equation*}
N + \frac{S(0) + 1}{4} = \frac{1}{2\pi i } \left[ \log \left(-S(+0)\right) -  \log \left(-S(+\infty)\right) \right].
\end{equation*}			
\end{enumerate}
\end{definition}

For $N \in \mathbb{N}$, let
\begin{equation*}
\Gamma_N := \left\lbrace \left(k_1, ... \,, k_N\right) \in \ii \mathbb{R}_+^N \quad : |k_1| > ...  >|k_N|>0 \right\rbrace.
\end{equation*}

The following theorem is the Marchenko Theorem in \cite{Marchenko}, which is stated for Dirichlet boundary condition, adapted to our setting with Robin boundary condition.

\begin{theorem}[Marchenko theorem]\label{Marchenko definitions}
Consider the mapping	
\begin{equation*}
\Sigma: L_{1,1} \left(\mathbb{R_+}\right) \to \bigcup_{N} \; S_N \times \mathbb{R}_+^N \times \Gamma_N
\end{equation*}	
defined by $\Sigma(V) := \left( S(k), \left(m_n\right)_{1,\,... \,, N},\left(k_n\right)_{1,\,... \,, N}\right)$ where
\begin{enumerate}[i)]
\item $S(k)$ denotes the scattering function defined in \eqref{scattering function},
\item $k_n$ denote the zeros of the Jost function defined in \eqref{Jost function},
\item $f(x,k_n) \in L^2(\mathbb{R}_+)$ are the eigenfunctions of $-\frac{d^2}{dx^2} + V$, with $f(x,k)$ being the solution of \eqref{Schrod equation} with condition \eqref{Jost solution condition},
\item \begin{equation*}
m_n= \int_0^{\infty} |f(x,k_n)|^2 dx, \qq k_n \in \mathbb{C}_+.
\end{equation*}

\end{enumerate}
Then the mapping $\Sigma$ is one-to-one and onto.
Moreover, for\footnote{In the case with Dirichlet boundary condition $G(x)$ is defined with a minus sign (see \cite{Marchenko,Korotyaev}).} (see \cite{Xiao, Marchenko}) 
\begin{align}\label{definition of G and G0}
G(x) = \frac{1}{2\pi} \int_{-\infty}^{\infty} (S(k) - 1)e^{\ii kx} dk, \qq G_0(x) = G(x) + \sum_{k_n \in \mathbb{C}_+} m_n^{-1} e^{-x|k_n|},
\end{align}
we can define the inverse mapping from  $ \left( S(k), \left(m_n\right)_{1,\,... \,, N},\left(k_n\right)_{1,\,... \,, N}\right)$ to $V \in L_{1,1}$ through
\begin{equation}\label{Potential recovery}
V(x) = -2 \frac{d}{dx} A(x,x),
\end{equation}
where $A(x,t)$ is the unique solution (see \cite{Marchenko}) for each $x>0$ of the Marchenko equation
\begin{align}\label{Marchenko equation}
A(x,t) = - G_0(x+t) - \int_{x}^{\infty} G_0(t+s) A(x,s) ds, \qq t \geq x.
\end{align}
\end{theorem}

\begin{remark}
Observe that the pair $\left(S(k), m_n\right)$ does not depend on $f_h(0)$ since $S(k) = - \frac{f_h(-k)}{f_h(k)}$ and $m_n = -\ii \frac{\dot{f}_h(k_n)}{f_h(-k_n)}$, and in the ratios the constant $f_h(0)$ gets cancelled.
\end{remark}
In the next lemma we show that the Jost function of our self-adjoint Schr{\"o}dinger problem with Robin boundary conditions and $V \in \mathbb{V}_{x_I}$ is of Cartwright class with $\rho_+(f_h)=0$ and $\rho_-(f_h)=2x_I$.

\begin{lemma}\label{Cartwright class lemma}
If $V \in \mathbb{V}_{x_I}$  then the Jost function \eqref{Jost function} is entire and of  exponential type. In particular, $f_h(k)$ satisfies the following conditions:
\begin{equation*}
\int_{\mathbb{R}}  \frac{\log^+ |f_h(k)| dk}{1 + k^2}<\infty, \quad \rho_+(f_h)=0, \; \rho_-(f_h)=2x_I.
\end{equation*}
In other words, $f_h$ is of Cartwright class (see Definition \ref{Cartwright class definition}).
\end{lemma}
\begin{proof}
In Theorem \ref{Analyticity and continuity of Jost function and solution} we proved that $f_h(k)$ is entire.
Let $x:=\re k$ and $y:=\im k$. Then by Corollary \ref{Corollary asympt expans phys sheet}, we can write the integral condition as
\begin{align*}
\int_{\mathbb{R}}  \frac{\log^+ |f_h(x)| dx}{1 + x^2} =\int_{\mathbb{R}}  \frac{ \log | \ii x + O(1)|}{1 + x^2}  dx < \infty.
\end{align*}
Using Corollary \ref{Corollary asympt expans phys sheet} we have
\begin{align*}
\rho_{+}(f_h)=\limsup_{y\to \infty} \frac{\log|-y + O(1)|}{y} = 0,
\end{align*}
while using Lemma \ref{uniform bounds on Jost function} we get
\begin{align}\label{formula rho -}
\rho_{-} (f_h) \leq \limsup_{y\to -\infty} \frac{2 x_I y}{y} = 2 x_I.
\end{align}
In order to prove the equality in \eqref{formula rho -}, we recall the definitions \eqref{definition of G and G0}, where if $G_0(x) = 0$ for $x>2dx_I$ for some $d\geq 1$, then $V=0$ for $x>x_I$. If $\rho_{-}(f_h(k))=2 d x_I < 2x_I$, then, thanks to the Jordan lemma, \eqref{scattering function} and the residue theorem, taking a contour on the upper half plane, we obtain
\begin{align}
&G(x) = \frac{1}{2 \pi} \int_{-\infty}^{+\infty} (S(k)-1)e^{\ii k x} dk = -\frac{1}{2 \pi} \int_{-\infty}^{+\infty}  \left(  \frac{f_h(k) + f_h(-k)}{f_h(k)} \right)e^{\ii kx} dk \nonumber \\
&=-\ii \sum_{j=1}^{N} \lim_{k\to k_j} \left( (k-k_j) \frac{f_h(k) + f_h(-k)}{f_h(k) - f_h(k_j)}\right) e^{\ii k x} =- \ii \sum_{j=1}^{N} \frac{f_h(-k_j)}{\dot{f}_h(k_j)}e^{\ii k_j x} \nonumber\\
&= - \sum_{j=1}^{N} \frac{1}{m_j} e^{\ii k_j x},\qq \text{if }x>2dx_I. \label{function G(x)}
\end{align}
Hence, $G_0(x)= G(x) + \sum_{j=1}^{N} \frac{1}{m_j} e^{\ii k_j x}$ is zero for $x>2dx_I$, so $V(x)=0$ for $x>x_I$, thus, $\rho_{-}(f_h(k))=2 x_I$ (see also \cite{Korotyaev}).
\end{proof}

\begin{remark}
The condition $\rho_-(f_h)=2x_I$ in the definition of the Cartwright class tells us that the resonances distribute in the unphysical sheet following a logarithmic curve. This is important because it implies that $\sum_{n} \frac{1}{k_n}$ is convergent and that $\prod_{n} \left(1 - \frac{k}{k_n}\right)$ converges to an entire function of exponential type.
\end{remark}

Lemma \ref{uniform bounds on Jost function} allows us to use the Hadamard factorization (Theorem \ref{Hadamard theorem}), where $m=0$ (see \cite[Theorem 2.3.6.]{FreilingYurko}), $E_{P}(z)= (1-z)e^z$ since the function is of exponential order one, $g(z)= az+b$ with $e^b=f_h(0)$ and $e^{(a + \sum_n \frac{1}{k_n})z}=e^{\ii z}$ by Corollary \ref{Corollary asympt expans phys sheet}. Hence, in our case, the Hadamard formula becomes 
\begin{align}\label{Hadarmard}
f_h(k) =  f_h(0)  e^{ik} \lim_{R \to \infty} \prod_{|k_n|\leq R} \left(1 - \frac{k}{k_n}\right),
\end{align}
where $k_n$ are the zeros of $f_h(k)$ counted with multiplicity.

\begin{remark}\label{Remark on constant fh(0)}
In the formula \eqref{Hadarmard} the constant $f_h(0)$ is uniquely determined by the resonances. It is possible to obtain $f_h(0)$ from the asymptotics obtained in Lemma \ref{uniform bounds on Jost function}, because the Jost function must satisfy $f_h(k)=\ii k +O(1)$ for large $k$ and changing $f_h(0)$ will change the asymptotics. Also Korotyaev in the Dirichlet case (see \cite[page 224 at the end of the proof of Theorem 1.1]{Korotyaev}) claims that the Jost function can be uniquely determined from the resonances.
\end{remark}
$f_h(k)$ is therefore determined uniquely by the resonances $k_n$ as explained in Remark \ref{Remark on constant fh(0)}.
Since $\sum_{n=1}^{\infty} \frac{1}{k_n}$ is absolutely convergent and since $k_j$ is a resonance if and only if $-\overline{k}_j$ is a resonance, then also $\sum_{n=1}^{\infty} \frac{1}{\overline{k}_n}$ is absolutely convergent. Furthermore, we have
\begin{align*}
\frac{1}{2} \sum_{n=1}^{\infty} \frac{\ii}{k_n} + \frac{1}{2}\sum_{n=1}^{\infty} \frac{-\ii}{\overline{k}_n} = \frac{1}{2} \sum_{n=1}^{\infty}  \frac{\ii (\overline{k}_n- k_n)}{|k_n|^2} = \sum_{n=1}^{\infty}  \frac{\im k_n}{|k_n|^2} 
\end{align*}
so the \textit{Blaschke condition} is fulfilled, that is,  $\sum_{n=1}^{\infty}  \frac{\im k_n}{|k_n|^2}$ is absolutely convergent.
From \eqref{Hadarmard}, differentiating $f_h(k)$ with respect to $k$, we obtain
\begin{equation}\label{Hadamard for scattering function}
\frac{d}{dk}(\log (f_h(k))) = \frac{\dot{f}_h(k)}{f_h(k)} = i + \lim_{R \to \infty} \sum_{|k_n|\leq R} \frac{1}{k - k_n}.
\end{equation} 
uniformly on compact subsets of $\mathbb{C} \backslash \left( \left\lbrace 0 \right\rbrace \cup \bigcup \left\lbrace k_n \right\rbrace\right)$.

Below we state the main theorem of this section, giving a complete characterization of the class of potentials $\mathbb{V}_{x_I}$. 

\begin{theorem}[Characterization]\label{Injectivity of P to W map}
For fixed $h$, the map $J_h : \mathbb{V}_{x_I} \to W_{x_I}$ defined by $J_h\left(V\right):=f_h$ is well-defined and bijective.
\end{theorem}
\begin{proof}
We extend the proof in \cite[Theorem 1.1]{Korotyaev} to our case.
First, for fixed $h\in \mathbb{R}$ we prove that the map $J_h$ is well-defined.

Let  $V\in \mathbb{V}_{x_I}$, then we need to prove that $f_h(k) \in W_{x_I}$.
Using Lemma \ref{uniform bounds on Jost function}, \ref{Lemma norming constants} and \ref{lemma fourier transform} we can see that $f_h(k)$ is real on $\ii \mathbb{R}$ and satisfies \eqref{Jost function mapped from the potential}.  From Theorem \ref{Analyticity and continuity of Jost function and solution} and Lemma \ref{Cartwright class lemma} we know that $f_h$ is in the Cartwright class with $\rho_+=0$ and $\rho_-=2x_I$. Then we can use Lemma \ref{lemma fourier transform} and by \eqref{Jost funct in terms of Fourier transf}, we get the form of \eqref{condition 1 class of Jost function} which satisfies Condition \ref{fourier transform of the Jost function} of $W_{x_I}$.
Condition \ref{condition 2 Jost function class} of $W_{x_I}$ is fulfilled by Lemma \ref{Lemma norming constants}, hence $f_h(k) \in W_{x_I}$ and $J_h$ is well-defined.

Consequently, $V \in \mathbb{V}_{x_I}$ uniquely determines $f_h(k) \in W_{x_I}$, which uniquely determines $\left( S(k), \left( m_j, k_j \right)_{j=1,\ldots,N} \right)$, which in turn uniquely determines $V \in L_{1,1}$ through the map $\Sigma$ of Marchenko theorem. Suppose now $f_h$ is the Jost function of  $V_1\in \mathbb{V}_{x_I}$ and $V_2 \in \mathbb{V}_{x_I}$, then $\left( S_1(k), \left( m_j, k_j \right)_{j=1,\ldots,N_1} \right) = \left( S_2(k), \left( m_j, k_j \right)_{j=1,\ldots,N_2} \right)$, and then using the map $\Sigma$ we deduce $V_1=V_2$ from Theorem \ref{Marchenko definitions}. Hence the map $J$ is injective.

We are left to prove that $J$ is surjective. Fix the scattering data $f_h(k) \in W_{x_I}$. We want to construct  $V  \in \mathbb{V}_{x_I}$ such that $J_h\left(V\right) = f_h$. We show this by proving that from $f_h(k)$ we can construct the scattering data $\left( S(k), \left( m_j, k_j \right)_{j=1,\ldots,N} \right)$ from \eqref{norming constant for eigenvalues} and \eqref{scattering function} and they satisfy the conditions of the Marchenko theorem.
We show that the scattering function satisfies the conditions (1), (2) and (3) of the Definition \ref{scattering function in the class}. 

(1) The scattering matrix $S(k) = - f_h(-k)/f_h(k)$ is continuous for $k \in \mathbb{R}$ and it is analytic everywhere excepts at the $k_n$, the zeros of $f_h(k)$. From the properties of the scattering phase, we can easily check that it holds that  $S(k)=\overline{S(-k)}=S^{-1}(-k)$ and, if $N$ is even, $S(0)=-(-1)^{\mathcal{N}_0(f_h)}$, where $\mathcal{N}_0(f_h)$ is the multiplicity of $0$ as a zero of $f_h$.

(2) From \eqref{bound first iteration Jost function} we have that 
\begin{align*}
&|f_h(k)-\ii k| \leq C, \qq \qq &k \in \overline{\mathbb{C}}_+,  \\
&|f_h(-k) + \ii k | \leq C e^{2x_I |\Im k|}, \qq &k \in \overline{\mathbb{C}}_+,
\end{align*}
which imply
\begin{align}\label{inequality scattering function}
\left| S(k) - 1 \right| &= \left| \frac{f_h(k) + f_h(-k)}{f_h(k)}\right|  \leq \frac{C}{|k|} \left(\left| f_h(k) - \ii k\right| + \left| f_h(-k) + \ii k\right|  \right) \nonumber  \\
&\leq \frac{ C_1 e^{2 x_I \Im k} + C_2}{|k|} \leq \frac{C e^{2 x_I \Im k}}{|k|}, \qq k \in \overline{\mathbb{C}}_+.
\end{align}  
Here and throughout $C$ denotes a positive constant that can change from line to line.
Using the Jordan lemma as in \eqref{function G(x)}, we get that $G_0(x)= G(x) + \sum_{j=1}^{N} \frac{1}{m_j} e^{\ii k_j x}$ is zero for $x>2x_I$. 
From \eqref{fourier transform of the Jost function} we have 
\begin{align*}
& - \frac{f_h(-k)}{f_h(k)}  - 1= -1 - \frac{-\ii k \left[ 1 + \frac{1}{2\ii k} \left(  \hat{F}(0) -  \hat{F}(-k) \right)\right]}{\ii k \left[ 1 - \frac{1}{2\ii k} \left(  \hat{F}(0) -  \hat{F}(k) \right)\right]} \nonumber \\
&= - \frac{-\frac{1}{2\ii k}\left(  2\hat{F}(0) - \hat{F}(k) -  \hat{F}(-k) \right) }{1 - \frac{1}{2\ii k} \left(  \hat{F}(0) -  \hat{F}(k) \right)} = \tfrac{1}{2\ii k} \left(  2\hat{F}(0) - \hat{F}(k) -  \hat{F}(-k) \right) + O(k^{-2})
\end{align*}
where $\hat{F}$ is continuous and bounded as $F \in  \mathbb{V}_{x_I}$.
We define the following functions 
\begin{equation*}
g_1 + g_2 = S(k) - 1, \quad g_1 = \frac{ 2\hat{F}(0) -\hat{F}(k) - \hat{F}(-k)  }{2\ii k}, \quad G_p := \frac{1}{2 \pi} \int_{\mathbb{R}} e^{ixk} g_p(k) dk 
\end{equation*}
for $p=1,2$. Then $G_2 \in L^2(\mathbb{R}) \cap L^{\infty}(\mathbb{R})$ since $g_2(k) = O(k^{-2})$. The function $G_1$ is odd in $x\in \mathbb{R}$ since $g_1$ is odd in $k$:
\begin{align*}
&G_1(-x) = - \frac{1}{2 \pi} \int_{-\infty}^{\infty} e^{-ixk} \frac{\hat{F}(k) + \hat{F}(-k) - 2\hat{F}(0)}{2\ii k} dk \\
&=   \frac{1}{2 \pi} \int_{+\infty}^{-\infty} e^{-ixk} \frac{\hat{F}(-k) + \hat{F}(k) - 2\hat{F}(0)}{-2\ii k} dk = - G_1(x) 
\end{align*}
and $G_1 \in L^2(\mathbb{R})$ because it is the Fourier transform of $g_1$ which is in $L^2(\mathbb{R})$. Since $|S - 1|=|g_1 +g_2|$ satisfies \eqref{inequality scattering function} and $g_2$ is bounded,
\begin{align*}
|g_1| \leq C \leq \frac{C e^{2 x_I \Im k}}{|k|}, \qq k \in \overline{\mathbb{C}}_+, \; |k| \gg1.
\end{align*}
Using the Jordan lemma, $G_1(x) = 0$ for $x > 2x_I$, so since $G_1$ is odd, $G_1(x)= 0$ for $|x|<2x_I$.
By $G_1$ being in $L^2$ and with compact support we get $G_1 \in L^1(\mathbb{R})$.
For $G'(x)$ we get
\begin{align}\label{G' equation}
&G'(x) = \frac{1}{2\pi} \int_{-\infty}^{\infty} e^{ixk} (ik)  \left(S(k) -1 \right) dk \nonumber \\
&=  - \frac{1}{2\pi} \int_{-\infty}^{\infty} e^{ixk} \left( \frac{\hat{F}(k) +  \hat{F}(-k)- 2\hat{F}(0) }{2}  + O(k^{-1})\right) dk
\end{align}
and using
\begin{align*}
&\frac{1}{2\pi} \int_{-\infty}^{\infty} e^{ixk} \hat{F}(k) dk = 	\frac{1}{2\pi} \int_{-\infty}^{\infty} e^{ixk} \left( \int_{0}^{x_I} e^{2\ii k y} F(y) dy\right) dk \\
&= 	\frac{1}{2\pi}  \int_{0}^{x_I} \int_{-\infty}^{\infty} e^{2\ii k (y+\frac{x}{2})} F(y) dk dy = \frac{F\left(- \frac{x}{2}\right)}{2},
\end{align*}
by the Fourier inversion formula, \eqref{G' equation} becomes
\begin{equation*}
G'(x) =- \frac{F\left( \frac{x}{2}\right) + F\left(- \frac{x}{2}\right) - 2 F(0)}{4}  + r(x) = - \frac{F\left( \frac{x}{2}\right) + F\left(- \frac{x}{2}\right)}{4} + r(x) 
\end{equation*}
where $r(x) \in L^2(\mathbb{R_+})$ is obtained through the Paley-Wiener theorem. Then $G'(x) \in L^1(\mathbb{R}_+,(1+x)dx)$, because $\supp G_0 \subset  \left[-2x_I, 2x_I \right]$.

(3) Since $f_h(k)$ is entire in the upper half plane, 
\begin{align}\label{formula Cauchy argument principle}
& \int_{\gamma} \frac{f'_h(z)}{f_h(z)} dz = \int_{\gamma} d(\log\left(f_h(z)\right)) = \lim_{R \to \infty}  \left( \int_{\gamma_R} d(\log\left(f_h(z)\right))  \right. \\
&\left. + \lim_{r \to \infty}  \int_{\gamma_r}  d(\log\left(f_h(z)\right)) \nonumber + \lim_{\epsilon \to 0} \int_{-R}^{-\epsilon} d(\log\left(f_h(z)\right)) + \lim_{\epsilon \to 0} \int^{R}_{\epsilon} d(\log\left(f_h(z)\right)) \right)
\end{align}
where $\gamma$ is a closed curve in the upper half plane made by a part that goes from $-R$ to $R$ passing around $z=0$ through a semi-circle $\gamma_r$ and an arc $\gamma_R$ in the upper-half plane. The integral over the big arc goes to zero because of the Jordan lemma, while the integral over the little arc gives a term $-\pi \ii \mathcal{N}_0(f_h)$. Hence we have
\begin{align}\label{formula integrals}
\lim_{\epsilon \to 0} \int_{-\infty}^{-\epsilon} d(\log\left(f_h(z)\right)) + \lim_{\epsilon \to 0} \int^{\infty}_{\epsilon} d(\log\left(f_h(z)\right)) = 2 \pi \ii \left(  \mathcal{N}(f_h) + \mathcal{N}_0(f_h)/2 \right).
\end{align}
We know that $S(0)=-(-1)^{\mathcal{N}_0(f_h)}$ and we know that zero is not an eigenvalue (see \cite[Theorem 2.3.6.]{FreilingYurko})., so $\mathcal{N}_0(f_h)=0$. So, we can write
\begin{align}\label{S(0) and N0}
\frac{S(0) + 1}{4} = \mathcal{N}_0(f_h)/2 = 0.
\end{align}
Computing the first two integrals of \eqref{formula integrals} we have
\begin{align}\label{integrals real line}
&\lim_{\epsilon \to 0, R\to \infty} \int_{-R}^{-\epsilon} d(\log\left(f_h(z)\right)) + \lim_{\epsilon \to 0} \int^{R}_{\epsilon} d(\log\left(f_h(z)\right)) =\lim_{\epsilon \to 0, R\to \infty} \left(  \log\left(f_h(+R)\right) \right. \nonumber\\ 
&\left. - \log\left(f_h(+\epsilon)\right) - \log\left(f_h(-R)\right) + \log\left(f_h(-\epsilon)\right) \right)
\end{align}
and since
\begin{align*}
- \log \left(-S(z) \right) = \log\left(f_h(z)\right) - \log\left(f_h(-z)\right)= \arg f_h(z) - \arg f_h(-z),
\end{align*}	
we finally obtain
\begin{align}\label{formula log of S}
&\lim_{\epsilon \to 0, R\to \infty} \int_{-R}^{-\epsilon} d(\log\left(f_h(z)\right)) + \lim_{\epsilon \to 0} \int^{R}_{\epsilon} d(\log\left(f_h(z)\right)) \nonumber\\
&= \log \left( -S(+0) \right) - \log \left( -S(+\infty) \right).
\end{align}
Inserting \eqref{formula log of S} and \eqref{S(0) and N0} into \eqref{formula integrals}, we obtain
\begin{align*}
\frac{1}{2 \pi \ii} \left(  \log \left( -S(+0) \right) - \log \left( -S(+\infty) \right) \right) = N + \frac{S(0) + 1}{4}.
\end{align*}
It follows that, all of the conditions of Definition \ref{scattering function in the class} are satisfied. The other conditions on $\left\lbrace m_j, k_j \right\rbrace_{j=1,\ldots,N}$ are implied by Condition \ref{condition 2 Jost function class} of the class $W_{x_I}$, hence the Marchenko theorem holds and there exists a unique $V \in L_{1,1}$ corresponding to the Jost function. We proved in (2) that $\supp G_0 \subset  \left[-2x_I, 2x_I \right]$, which implies $V=0$ for $x>x_I$. 

If $\rho_-(f_h) = 2t$, where $t:=x_I - \epsilon_0$ with $$\epsilon_0:=\inf\{\epsilon>0: \left|(x_I - \epsilon, x_I) \cap \supp V \right|>0\},$$ then $V \in \mathbb{V}_{t}$, but since $\rho_{-}(f_h) = 2x_I$ as explained in the proof of Lemma \ref{lemma fourier transform}, then $V \in \mathbb{V}_{x_I}$ is the unique potential corresponding to the Jost function $f_h \in W_{x_I}$.
\end{proof}
Theorem \ref{Injectivity of P to W map} suggests an algorithm that enables us to reconstruct the unique potential from a set of resonances.
\begin{algorithm}
Starting from a set of eigenvalues and resonances $\left\lbrace k_j \right\rbrace_{1}^{\infty}$ we can retrieve the potential $V_{\omega}(x)$ using the following algorithm:
\begin{itemize}
\item Construct the Jost function from \eqref{Hadarmard} as 
\begin{align}\label{Hadamard fact formula}
f_h(k) =  f_h(0)  e^{ik} \lim_{R \to \infty} \prod_{|k_n|\leq R} \left(1 - \frac{k}{k_n}\right),
\end{align}
where $f_h(0)$ is determined so that $f_h(k)= \ii k + O(1)$ as $k \to \infty$.
\item Use $\left\lbrace k_j \right\rbrace_{1}^{\infty}$ and $f_h(k)$ to construct the scattering data $\left( S(k), \left\lbrace m_j, k_j \right\rbrace_{j=1,\ldots,N} \right)$ from equations \eqref{scattering function} and \eqref{norming constant for eigenvalues} as
\begin{align*}
&S(k) = - e^{-2\ii k} \prod_{n\geq 1}^{\infty} \left( \frac{k_n + k}{k_n - k}\right),\\
&m_j=\frac{e^{-2 |k_j|}}{2 |k_j|} \prod_{n\geq 1, n \neq j} \left(\frac{k_n - k_j}{k_n + k_j}\right), \qq j=1, ... \,, N.
\end{align*}
\item Use the scattering data $\left( S(k), \left\lbrace m_j, k_j \right\rbrace_{j=1,\ldots,N} \right)$ to construct $G_0(x)$ in \eqref{definition of G and G0}.
\item  Solve \eqref{Marchenko equation} for $A(x,t)$.
\item  Obtain the potential from \eqref{Potential recovery}.
\end{itemize}
\end{algorithm}

After the recovery of the potential $V_{\omega}(x)$, where we added the subscript $\omega$ because it is found for every fixed value of $\omega$, we need to recover the shear modulus $\hat{\mu}(x)$, which, physically, is more interesting. This can be done from the knowledge of the potential at two different values $\omega_1$ and $\omega_2$, with $\omega_1 \neq \omega_2$, as we present in the following theorem. A similar idea of two frequencies reconstruction can also be seen in \cite{Iantchenko} with respect to the Weyl matrix and Jost function matrix as in Proposition V.1 and Theorem V.1.

\begin{theorem}\label{Theorem on Lame parameter mu}
Let $V_{\omega_1}(x)$ and $V_{\omega_2}(x)$ be the potential at the frequencies $\omega_1$ and $\omega_2$, with $\omega_1 \neq \omega_2$, then the shear modulus can be retrieved by the following formula
\begin{equation}\label{Recovery of Lame parameter}
\hat{\mu}(x) = \frac{\hat{\mu}_I  \left( \omega^2_1 - \omega^2_2 \right)}{\omega^2_1 - \omega^2_2 - \hat{\mu}_I \left( V_{\omega_1}(x) - V_{\omega_2}(x) \right)}.
\end{equation}
\end{theorem}
\begin{proof}
In \eqref{Definition of the potential} we defined the potential as
\begin{equation*}
V_{\omega}= \frac{(\sqrt{\hat{\mu}})''}{\sqrt{\mu}} - \frac{1}{\hat{\mu}}\omega^2 + \frac{1}{\hat{\mu}_I}\omega^2.
\end{equation*}
Then the potential difference at two different frequencies, respectively $\omega_1$ and $\omega_2$ is
\begin{equation*}
V_{\omega_1}(x) - V_{\omega_2}(x) = \left(\frac{1}{\hat{\mu}_I} - \frac{1}{\hat{\mu}(x)}  \right) (\omega^2_1 - \omega^2_2)
\end{equation*}
which leads to \eqref{Recovery of Lame parameter}.
\end{proof}

\begin{remark}
In this paper, we do not treat the stability of the resonances. For more details, we refer the reader to \cite{Korotyaev_stability, MarlettaWeikard} who treat the Dirichlet case.
\end{remark}

\section{acknowledgement}

I want to thank my former supervisor Alexei Iantchenko for having introduced me to the topic of inverse resonance problems and for interesting discussions.

\bibliographystyle{amsplain}
\bibliography{MonographyCitations}

\end{document}